\documentclass[10pt]{article}
\usepackage{amsmath, amsthm}\usepackage{enumerate}
\usepackage{amssymb}\usepackage{bbold}
\usepackage{color}
\newtheorem{theorem}{Theorem}
\newtheorem{definition}{Definition}
\newtheorem{exam}{Example}

\newtheorem{lem}{Lemma}

\DeclareMathOperator{\convo}{\xrightarrow[]{o}}
\DeclareMathOperator{\convso}{\xrightarrow[]{so}}

\DeclareMathOperator{\convn}{\xrightarrow[]{\|\cdot\|}}
\DeclareMathOperator{\convsn}{\xrightarrow[]{\|\cdot\|_\infty}}

\makeatletter
\renewcommand{\subsection}{\@startsection{subsection}{1}
{0pt}{3.25ex plus 1ex minus.2ex}{-1em}{\normalfont\normalsize\bf}}
\makeatother
\begin{document}

\title{{\bf On KB and Levi operators in Banach lattices}}
\date{}
\maketitle
\author{\centering{{Eduard Emelyanov$^{1}$\\ 
\small $1$ Sobolev Institute of Mathematics, Novosibirsk, Russia}

\abstract{We prove that an order continuous Banach lattice $E$ is a \text{\rm KB}-space 
if and only if each positive compact operator on $E$ is a \text{\rm KB} operator. We give
conditions on quasi-\text{\rm KB} (resp., quasi-Levi) operators to be \text{\rm KB} 
(resp., Levi), study norm completeness and domination for these operators,
and show that neither \text{\rm KB} nor \text{\rm Levi} operators are stable
under rank one perturbations.


\large

\bigskip

\section{Introduction}

Several operator versions of Banach lattice properties 
like to be a KB-space were  
studied recently (see, for example \cite{AlEG,BA,GE,ZC}). 
They emerge via redistribution of topological 
and order properties between the domain and range. As 
the order convergence is not topological in infinite 
dimensional vector lattices \cite{DEM}, important 
operator versions use to occur if order and 
norm convergences are both involved. We investigate
(quasi-) \text{\rm KB} and (quasi-) Levi operators.

Throughout the paper vector spaces are real and operators are linear and bounded. 
Letters $X$ and $Y$ stand for a Banach spaces, $E$ and $F$ for Banach lattices,
$\text{\rm L}(X,Y)$ (resp., $\text{\rm L}_{\text{\rm FR}}(X,Y)$, \text{\rm K}(X,Y), 
\text{\rm W}(X,Y)) for the space of all bounded 
(resp., finite rank, compact, weakly compact) operators from $X$ to $Y$, $I_X$ for 
the identity operator in $X$, and $\text{\rm L}_r(E,F)$ for the space of regular 
operators in $\text{\rm L}(E,F)$.
We write $y_\alpha\downarrow 0$ for a net $(y_\alpha)$ in $E$, whenever 
$y_{\alpha'}\le y_\alpha$ for all $\alpha'\ge\alpha$ and $\inf_E y_\alpha=0$.
In the present paper, we use the following definition of order convergence.

\begin{definition}\label{order convergence} 
{\em
Let $(x_\alpha)$ be a net in $E$ and $x\in E$.
\begin{enumerate}[$a)$]
\item\ 
$(x_\alpha)$ order converges to $x$ 
(briefly, $x_\alpha\stackrel{\text{\rm o}}{\to} x$) 
whenever there exists a net $(y_\beta)$ in $E$ such that 
$y_\beta\downarrow 0$ and, for each $\beta$ 
there is an $\alpha_\beta$ with $|x_\alpha-x|\le y_\beta$ for
all $\alpha\ge\alpha_\beta$. 
\item\ 
$(x_\alpha)$ strong order converges to $x$ 
($x_\alpha\stackrel{\text{\rm so}}{\to} x$) 
whenever $|x_\alpha-x|\le y_\alpha\downarrow 0$
for some net $(y_\alpha)$ in $E$. 
\end{enumerate}
}
\end{definition}
\noindent
Despite our definition of order convergence is standard now-days, it differs with one 
used in the textbooks \cite{AB,Mey}, where the so-convergence was called
order convergence. Indeed, so-convergence is equivalent to o-convergence 
in Dedekind complete vector lattices (see \cite[Prop.1.5]{AS}).  

\medskip
The next definition is an adopted variant of \cite[Def.1.1]{AlEG}, 
whereas item $e)$ is new.

\begin{definition}\label{order-to-topology} 
{\em
An operator $T\in\text{\rm L}(E,Y)$ is:
\begin{enumerate}[$a)$]
\item   
\text{\rm KB} (\text{\rm $\sigma$-KB}) if, for every
increasing net (sequence) $(x_\alpha)$ in $(B_E)_+$,
there exists an $x\in E$ such that $Tx_\alpha\stackrel{\|\cdot\|}{\to}Tx$.
The set of all such operators is denoted by $\text{\rm L}_{\text{\rm KB}}(E,Y)$ 
(resp., $\text{\rm L}^\sigma_{\text{\rm KB}}(E,Y)$).
\item   
\text{\rm quasi-KB} (\text{\rm quasi-$\sigma$-KB}) if $(Tx_\alpha)$
is norm convergent for every increasing net (sequence) $(x_\alpha)$ in $(B_E)_+$.
The set of all such operators is denoted by $\text{\rm L}_{\text{\rm qKB}}(E,Y)$ 
(resp., $\text{\rm L}^\sigma_{\text{\rm qKB}}(E,Y)$).
\end{enumerate}}
{\em
\noindent
An operator $T\in\text{\rm L}(E,F)$ is:
\begin{enumerate}[]
\item[$c)$] \  
\text{\rm Levi} (\text{\rm $\sigma$-Levi}) if, for every 
increasing net (sequence) $(x_\alpha)$ in $(B_E)_+$,
there exists an $x\in E$ with $Tx_\alpha\stackrel{\text{\rm o}}{\to} Tx$.
The set of all such operators is denoted by $\text{\rm L}_{\text{\rm Levi}}(E,F)$ 
(resp., $\text{\rm L}^\sigma_{\text{\rm Levi}}(E,F)$).
\item[$d)$] \  
\text{\rm quasi-Levi} (\text{\rm quasi-$\sigma$-Levi}) if $T$ takes 
increasing nets (sequences) in $(B_E)_+$ to o-Cauchy nets.
The set of all such operators is denoted by $\text{\rm L}_{\text{\rm qLevi}}(E,F)$ 
(resp., $\text{\rm L}^\sigma_{\text{\rm qLevi}}(E,F)$).
\item[$e)$] \  
\text{\rm complete-quasi-Levi} (\text{\rm complete-quasi-$\sigma$-Levi}) 
if $(Tx_\alpha)$ is order convergent for every increasing net 
(sequence) $(x_\alpha)$ in $(B_E)_+$. The set of all such operators 
is denoted by $\text{\rm L}_{\text{\rm c-qLevi}}(E,F)$ 
(resp., $\text{\rm L}^\sigma_{\text{\rm c-qLevi}}(E,F)$).
\end{enumerate}}
\end{definition}
\medskip
\noindent
Quasi-\text{\rm KB} operators were called \text{\rm KB} operators in \cite{BA,GE}.
They agree with b-weakly compact operators of \cite{AA} (see \cite[Prop.2.1]{ZC}).
By \cite[Prop.1.2]{AlEG}, 
$\text{\rm L}^\sigma_{\text{\rm qKB}}(E,Y)=\text{\rm L}_{\text{\rm qKB}}(E,Y)$.
Clearly, 
$$
\begin{matrix}
   \text{\rm L}_{\text{\rm Levi}}(E,F) & \subseteq &   \text{\rm L}_{\text{\rm c-qLevi}}(E,F)
                      & \subseteq &   \text{\rm L}_{\text{\rm qLevi}}(E,F) \\
 \text{\small $|\bigcap$}  &  & \text{\small $|\bigcap$} &   & \text{\small $|\bigcap$}  \\
  \text{\rm L}^\sigma_{\text{\rm Levi}}(E,F) &\subseteq &  \text{\rm L}^\sigma_{\text{\rm c-qLevi}}(E,F)    & \subseteq &   \text{\rm L}^\sigma_{\text{\rm qLevi}}(E,F).
\end{matrix}
$$
While replaced o-convergence by so-convergence, the comp\-lete-quasi-($\sigma$-) Levi 
operators turn to ($\sigma$-) Levi operators of Zhang and Chen \cite{ZC}.
As so-convergence implies o-conver\-gence, then each ($\sigma$-) Levi 
operator  $T:E\to F$ of \cite{ZC} is complete-quasi-($\sigma$-) Levi in the sense of
Definition \ref{order-to-topology}. By \cite[Prop.1.5]{AS}, the converse is true
when $F$ is Dedekind complete. 
 
There are $\sigma$-Levi operators which are not even quasi-Levi.
Indeed, consider a Banach lattice $E=\ell^\infty_\omega(\Gamma)$ of 
real-valued bounded countably
supported functions on an uncountable set $\Gamma$. 
Clearly $I_E\in\text{\rm L}^\sigma_{\text{\rm Levi}}(E)$. Take a bounded 
increasing net $(1_\alpha)_{\alpha\in{\cal P}_{fin}(\Gamma)}$, where
${\cal P}_{fin}(\Gamma)$ is the set of all finite subsets of $\Gamma$
directed by inclusion, and $1_\alpha$  is the indicator function of 
$\alpha$. Since the net $(1_\alpha)$ is not o-Cauchy, 
$I_E\notin\text{\rm L}_{\text{\rm qLevi}}(E)$  
(cf. also \cite[Ex.2.6]{ZC}). 

The next example shows that a quasi-KB (quasi-Levi) operator 
doesn't have to be \text{\rm KB} (resp., complete-quasi-Levi).

\begin{exam}\label{quasi KB yet not KB}
{\em
First we show that the inclusion $\text{\rm L}_{\text{\rm KB}}(E)\subseteq\text{\rm L}_{\text{\rm qKB}}(E)$
can be proper.
\begin{enumerate}[$a)$]
\item
Let $E=C[0,1]\oplus L_1[0,1]$. Define an operator $T\in\text{\rm L}(E)$ by
$T\bigl((\phi,\psi)\bigl):=(0,\phi)$ for $\phi\in C[0,1]$ and $\psi\in L_1[0,1]$. 
Clearly, $T\in\text{\rm L}_{\text{\rm qKB+}}(E)$. But $T$ is not even \text{\rm $\sigma$-KB}.
Indeed, consider a continuous function $\phi_n$ that equals 
to 1 on $\bigl[0,\frac{1}{2}-\frac{1}{2^n}\bigl]$,
to 0 on $\bigl[\frac{1}{2},1\bigl]$, and is linear otherwise.
Let $f_n:=(\phi_n,0)$. Then $(B_E)_+\ni f_n\uparrow$
and $Tf_n\convn(0,g)$, where $g\in L_1[0,1]$ is the indicator 
function of $\bigl[0,\frac{1}{2}\bigl]$.
Since $g\not\in C[0,1]$, there is no element $f\in E$ satisfying
$Tf=(0,g)$, and hence $T\notin\text{\rm L}^\sigma_{\text{\rm KB}}(E)$.
\end{enumerate}}
{\em
\noindent
Next we show that the inclusion 
$\text{\rm L}_{\text{\rm c-qLevi}}(E)\subseteq\text{\rm L}_{\text{\rm qLevi}}(E)$
can also be proper.
\begin{enumerate}[$b)$]
\item
Consider the operator $I_c$,
where $c$ is the Banach lattice of all convergent real sequences. 
Since each bounded positive increasing 
net in $c$ is \text{\rm o}-Cauchy, $I_c\in\text{\rm L}_{\text{\rm qLevi}}(c)$.
Denote elements of $c$ by $\sum_{n=1}^\infty a_n \cdot e_n$,
where $e_n$ is the n-th unit vector of $c$ and $(a_n)$ converges in $\mathbb{R}$.
A bounded increasing sequence $I_cf_n=f_n=\sum_{k=1}^n e_{2k-1}$ in $c_+$ is not 
o-convergent. Thus, $I_c\notin\text{\rm L}^\sigma_{\text{\rm c-qLevi}}(c)$.
\end{enumerate}}
\end{exam}

\medskip
We shall use the following two elementary lemmas.

\begin{lem}\label{prop2}
\begin{enumerate}[$i)$]
The following holds.
\item
$\text{\rm L}_{\text{\rm qKB}}(E,Y)$, $\text{\rm L}_{\text{\rm c-qLevi}}(E,F)$,
and $\text{\rm L}_{\text{\rm qLevi}}(E,F)$ are vector spaces.
\item
If $E$ is a \text{\rm KB}-space then
$\text{\rm L}_{\text{\rm KB}}(E,Y)=\text{\rm L}(E,Y)$.
\item
$\text{\rm L}_{\text{\rm KB+}}(E,F)\subseteq\text{\rm L}_{\text{\rm Levi+}}(E,F)$.
\item
$\text{\rm L}_{\text{\rm qKB+}}(E,F)\subseteq\text{\rm L}_{\text{\rm c-qLevi+}}(E,F)$.
\item
$\text{\rm K}_+(E,F)\subseteq\text{\rm L}_{\text{\rm qKB+}}(E,F)$.
\end{enumerate}
\end{lem}

\begin{proof}
Items $i)$ and $ii)$ are trivial. Items $iii)$ and $iv)$
follow from the fact that each norm convergent increasing net in F converges 
in order to the same limit. 

$v)$\ 
Let $T\in\text{\rm K}_+(E,F)$ and let $(x_\alpha)$ be an 
increasing net in $(B_E)_+$. Then $(Tx_\alpha)$ has a subnet 
$(Tx_{\alpha_\beta})$ satisfying $\left\|Tx_{\alpha_\beta}-y\right\|\to 0$ 
for some $y \in F$. Since $Tx_\alpha\uparrow$ then $Tx_\alpha\convn y$. 
Therefore $T\in\text{\rm L}_{\text{\rm qKB+}}(E,F)$.
\end{proof}
\noindent
By Lemma \ref{prop2}, $\text{\rm L}_{\text{\rm KB}}(E,Y)$ is complete in 
the operator norm for each \text{\rm KB}-space $E$, and by Theorem \ref{qKB-Banach-space}, 
$\text{\rm L}_{\text{\rm qKB}}(E,Y)$ is a norm-complete vector space.
Example \ref{KB-not-compl} shows that $\text{\rm L}_{\text{\rm KB}}(c_0)$ 
and $\text{\rm L}_{\text{\rm Levi}}(c_0)$ are not norm-complete. 

\medskip
In general, neither 
$\text{\rm L}_{\text{\rm KB}}(E,Y)$ nor $\text{\rm L}_{\text{\rm Levi}}(E,F)$ are
closed under the addition (see Examples \ref{KB are not VS} and \ref{Levi are not VS}). 
Coupled with Lemma \ref{prop2}, the next lemma gives some 
positive information in this direction.

\begin{lem}\label{prop1}
\begin{enumerate}[$i)$]
The following holds.
\item
$\text{\rm L}_{\text{\rm FR}}(E,Y)\subseteq\text{\rm L}_{\text{\rm KB}}(E,Y)$.
\item
$\text{\rm L}_{\text{\rm FR}}(E,F)\subseteq\text{\rm L}_{\text{\rm Levi}}(E,F)$.
\end{enumerate}
\end{lem}

\begin{proof}
$i)$\
Let $T:E\to Y$ be a finite rank operator. 
WLOG, we may identify $Y$ with $TE$ and the latter with $\mathbb{R}^n$. Then 
$$
   T = \sum_{k=1}^n f_k \otimes y_k 
   \  \text{for} \   y_1, \dots, y_n \in \mathbb{R}^n
   \ \text{and} \   f_1, \dots, f_n \in E'.
$$
Since $E'$ is Dedekind complete, functionals $f_1,...,f_n$ are regular, and hence
$$
   T = \sum_{k=1}^n (f_k^+ \otimes y_k^+ + f_k^- \otimes y_k^-)-
   \sum_{k=1}^n (f_k^+ \otimes y_k^- + f_k^- \otimes y_k^+).
$$
Let $(x_\alpha)$ be an increasing net in $(B_E)_+$. The nets
$$
   T_1x_\alpha = \Bigl(\sum_{k=1}^n (f_k^+ \otimes y_k^+ + f_k^- \otimes y_k^-)\Bigl)(x_\alpha)
$$
and 
$$
   T_2x_\alpha = \Bigl(\sum_{k=1}^n (f_k^+ \otimes y_k^- + f_k^- \otimes y_k^+))\Bigl)(x_\alpha)
$$
are both increasing and bounded in $\mathbb{R}^n=TE$. 
Then $T_1x_\alpha\convn z_1$ and $T_2x_\alpha\convn z_2$ for some $z_1,z_2\in TE$.
Thus
$$
   Tx_\alpha=(T_1x_\alpha-T_2x_\alpha)\convn(z_1-z_2)\in TE-TE=TE.
$$
Pick an $x\in E$ such that $Tx=z_1-z_2$. Then $Tx_\alpha\convn Tx$.
We conclude $T\in\text{\rm L}_{\text{\rm KB}}(E,Y)$.

\medskip
$ii)$\
Let $T\in\text{\rm L}_{\text{\rm FR}}(E,F)$, say
$$
   T = \sum_{k=1}^n f_k \otimes y_k 
   \  \text{for} \   y_1, \dots, y_n \in TE
   \ \text{and} \   f_1, \dots, f_n \in E'.
$$
Denote
$$
   T_1:=\sum_{k=1}^n f_k^+ \otimes y_k \ \ \text{\rm and} \ \ 
   T_2:=\sum_{k=1}^n f_k^- \otimes y_k. 
$$
Let $(x_\alpha)$ be an increasing net in $(B_E)_+$. The nets $f_k^+(x_\alpha)$ 
and $f_k^-(x_\alpha)$ are increasing and bounded in $\mathbb{R}_+$. 
Thus, $f_k^+(x_\alpha)\to a_k\in\mathbb{R}_+$ and 
$f_k^-(x_\alpha)\to b_k\in\mathbb{R}_+$, and hence
$$
   T_1x_\alpha=\sum_{k=1}^n f_k^+(x_\alpha) y_k\convn\sum_{k=1}^n a_k y_k
$$
and
$$
   T_2x_\alpha=\sum_{k=1}^n f_k^-(x_\alpha) y_k\convn\sum_{k=1}^n b_k y_k.
$$
Since $\dim(TE)<\infty$ then
$$
   T_1x_\alpha\convo\sum_{k=1}^n a_k y_k\ \ \ \text{\rm and} \ \ \ 
   T_2x_\alpha\convo\sum_{k=1}^n b_k y_k,
$$
and hence
$Tx_\alpha=(T_1x_\alpha-T_2x_\alpha)\convo\sum_{k=1}^n(a_k-b_k)y_k\in TE$.
Take an $x\in E$ such that $Tx=\sum_{k=1}^n(a_k-b_k)y_k$. Then $Tx_\alpha\convo Tx$.
We conclude $T\in\text{\rm L}_{\text{\rm Levi}}(E,F)$.
\end{proof}

\medskip

\noindent
The next example strengthens Example \ref{quasi KB yet not KB}  
by showing that even the inclusions 
$$
   \text{\rm L}_{\text{\rm KB}}(E)\cap\text{\rm K}_+(E)\subseteq
   \text{\rm L}_{\text{\rm qKB}}(E)\cap\text{\rm K}_+(E),
$$
$$
   \text{\rm L}_{\text{\rm Levi}}(E)\cap\text{\rm K}_+(E)\subseteq
   \text{\rm L}_{\text{\rm c-qLevi}}(E)\cap\text{\rm K}_+(E)
$$
are proper in general.

\begin{exam}\label{Example1 c_0}
{\em
Let $(\alpha_n)_{n=1}^\infty\in c_0\setminus c_{00}$. Define $T\in\text{\rm L}(c_0)$ by
$$
   T\left( \sum_{n=1}^\infty a_n e_n\right)  = \sum_{n=1}^\infty(\alpha_n a_n) e_n. 
   \eqno(1)
$$
Then $T\in\text{\rm K}_+(c_0)$, and hence 
$T\in\text{\rm L}_{\text{\rm qKB}}(c_0)\bigcap\text{\rm L}_{\text{\rm c-qLevi}}(c_0)$
by Lemma \ref{prop2}. 
Take an increasing sequence 
$x_n: = \sum\limits_{k=1}^n e_k$ in $(B_{c_0})_+$. The sequence 
$(Tx_n) =\Bigl(\sum\limits_{k=1}^n\alpha_k e_k\Bigl)$ converges in norm
and in order to $\sum\limits_{k=1}^{\infty}\alpha_ke_k\in c_0$, 
however there is no element $\sum\limits_{k=1}^\infty a_k e_k$ of $c_0$ with
$T\left(\sum\limits_{k=1}^\infty a_k e_k \right)=\sum\limits_{k=1}^\infty\alpha_ke_k$.
Indeed, would such an element $x=\sum\limits_{k=1}^\infty a_k e_k\in c_0$ with
$Tx=\sum\limits_{k=1}^\infty\alpha_ke_k$ exist, it must satisfies $a_k=1$ for all $k$
such that $\alpha_k\ne 0$, which is absurd. 
Therefore $T$ is neither \text{\rm KB} nor $\sigma$-Levi operator.

Furthermore, slightly modified arguments work also for operators $T\in\text{\rm L}(c)$
and $T\in\text{\rm L}(c,c_0)$ defined by the same formula (1) for 
a null sequence $(\alpha_n)$ in $\mathbb{R}$ consisting of non-zero elements.
To see this, it suffices to replace $x_n=\sum\limits_{k=1}^n e_k$ in $(B_{c_0})_+$ by 
the increasing sequence
$x_n=\sum\limits_{k=1}^n e_{2k}$ in $(B_{c})_+$.
}
\end{exam}

\medskip
\noindent
Example \ref{Example1 c_0} is serving also a counter-example 
to \cite[Prop.3.5]{AlEG}, where it is faultily claimed that every weakly 
compact positive operator between Banach lattices is Levi.

\medskip
The paper is organized as follows. In Section 2 we study (quasi-) \text{\rm KB} operators 
and their relations with (weakly) compact operators. Section 3 is devoted to (quasi-) Levi operators.

We refer the reader to Aliprantis and Burkinshaw \cite{AB} and Meyer-Nieberg
\cite{Mey} for unexplained terminology and notation concerning Banach lattices.

\section{KB operators}

As $\text{\rm L}_{\text{\rm KB}}(E,Y)=\text{\rm L}(E,Y)$ for every \text{\rm KB}-space $E$,
\text{\rm KB} operators are not necessarily order boun\-ded (see, for example \cite[Exer.10, p.289]{AB}).
Positive \text{\rm KB} operators need not to be weakly compact 
(e.g., $I_{E}\in\text{\rm L}_{\text{\rm KB+}}(\ell^1)\setminus\text{\rm W}(\ell^1)$), whereas  
Example \ref{Example1 c_0} presents a positive compact operator on $c_0$ that is not \text{\rm KB}.
In this section, after brief discussion of general properties of (quasi-) \text{\rm KB} operators,
we study some relations between them and (weak) compact operators. 
We show that \text{\rm KB} operators need not to be closed under the addition.
In the end of the section, we discuss the domination property for (quasi-) \text{\rm KB} operators. 

\subsection{}
We show that quasi-\text{\rm KB} operators are norm-closed. 

\begin{lem}\label{L2}
Let $\text{\rm L}_{\text{\rm qKB}}(E,Y)\ni 
T_\gamma\stackrel{\|\cdot\|}{\longrightarrow}T\in\text{\rm L}(E,Y)$. 
Then $T\in\text{\rm L}_{\text{\rm qKB}}(E,Y)$.
\end{lem}

\begin{proof}
Let $(x_\alpha)$ be an increasing net in $(B_{E})_+$. 
Take $\varepsilon>0$ and pick a $\gamma_0$ satisfying 
$\|T_{\gamma_0}-T\|\le\varepsilon$.
Choose an $\alpha_0$ such that 
$\|T_{\gamma_0}x_{\alpha^{'}}-T_{\gamma_0}x_{\alpha^{''}}\|\le\varepsilon$
for $\alpha^{'}, \alpha^{''}\ge\alpha_0$. Then 
$$
   \|Tx_{\alpha^{'}} - Tx_{\alpha^{''}} \| \le 
$$
$$
   \|Tx_{\alpha^{'}} - T_{\gamma_0} x_{\alpha^{'}} \|+ 
   \|T_{\gamma_0} x_{\alpha^{'}} - T_{\gamma_0} x_{\alpha^{''}}\|+ 
   \|T_{\gamma_0} x_{\alpha^{''}} - Tx_{\alpha^{''}} \| \le
$$
$$ 
   \|T_{\gamma_0} - T\| \|x_{\alpha^{'}}\| +  
   \|T_{\gamma_0} x_{\alpha^{'}} - T_{\gamma_0} x_{\alpha^{''}} \|+ 
   \|T_{\gamma_0} - T\| \|x_{\alpha^{''}}\|    \le 3\varepsilon
$$
for all $\alpha^{'}, \alpha^{''}\ge\alpha_0$, and hence $T$ is  quasi-\text{\rm KB}. 
\end{proof}

\begin{theorem}\label{qKB-Banach-space}
The set  $\text{\rm L}_{\text{\rm qKB}}(E,Y)$ is a Banach space under the operator norm. 
\end{theorem}

\begin{proof}
By Lemma \ref{prop2}, $\text{\rm L}_{\text{\rm qKB}}(E,Y)$ is a vector space.
By Lemma \ref{L2}, $\text{\rm L}_{\text{\rm qKB}}(E,Y)$ is a closed subspace of 
$\text{\rm L}(E,Y)$, and hence is a Banach space.
\end{proof}

\noindent
The following example shows that $\text{\rm L}_{\text{\rm KB}}(c_0)$ and 
$\text{\rm L}_{\text{\rm Levi}}(c_0)$ are not closed in $\text{\rm L}(c_0)$ 
under the operator norm.

\begin{exam}\label{KB-not-compl}
{\em
Consider $T\in\text{\rm L}(c_0)$
given by $Tx=\sum_{k=1}^\infty \frac{x_k}{k} e_k$ for $x\in c_0$,
and define a sequence $(T_n)$ of operators in $\text{\rm L}(c_0)$ by
$T_nx=\sum_{k=1}^n\frac{x_k}{k} e_k$. Then, we have the following.
\begin{enumerate}[$a)$]
\item\
Trivially,
$\text{\rm L}_{\text{\rm FR}}\bigl(c_0\bigl)\ni T_n\convn T$.
\item\
By Lemma \ref{prop1}, 
$T_n\in{\rm L}_{\text{\rm KB}}\bigl(c_0\bigl)\bigcap{\rm L}_{\text{\rm Levi}}\bigl(c_0\bigl)$.
\item\
By Example \ref{Example1 c_0}, 
$T\notin{\rm L}_{\text{\rm KB}}\bigl(c_0\bigl)\bigcup{\rm L}_{\text{\rm Levi}}\bigl(c_0\bigl)$.
\item\
It follows from $b)$ and $c)$ that both ${\rm L}_{\text{\rm KB}}\bigl(c_0\bigl)$ 
and ${\rm L}_{\text{\rm Levi}}\bigl(c_0\bigl)$ are not norm-closed in ${\rm L}\bigl(c_0\bigl)$.
\end{enumerate}
}
\end{exam}

\subsection{}
Note that both sets $\text{\rm L}_{\text{\rm KB}}(E,Y)$ 
and $\text{\rm L}_{\text{\rm Levi}}(E,F)$ are obviously closed under the scalar multiplication.
By Lem\-ma \ref{prop2}, $\text{\rm L}_{\text{\rm qKB}}(E,Y)$, $\text{\rm L}_{\text{\rm c-qLevi}}(E,F)$,
and $\text{\rm L}_{\text{\rm qLevi}}(E,F)$ are vector spaces.
The situation with \text{\rm KB} and \text{\rm Levi} operators is quite different.
The following example shows that even perturbation by a positive rank one operator
of a positive \text{\rm KB} operator belonging to norm closure of finite rank operators
need not to be a  \text{\rm KB} operator. 
For \text{\rm Levi} operators, see Example \ref{Levi are not VS} in Section 3.

\begin{exam}\label{KB are not VS}
{\em
Let $f\in(\ell^\infty)'$ be a Banach limit.
Define operators $S,T\in\text{\rm L}(\ell^\infty)$ by
$$
   S\mathbf{a} =
   \sum_{n=2}^\infty\Bigl(\sum_{k=1}^{n-1}\frac{a_k}{2^k}\Bigl) e_n, \ \ \text{\rm and} \ \ \
   T\mathbf{a} = f(\mathbf{a})e_1,
$$
where $\mathbf{a}=\sum_{n=1}^\infty a_n e_n\in\ell^\infty$. Operator $T$ is 
of rank one, and hence is a \text{\rm KB} operator by Lemma \ref{prop1}. 
If $(B_{\ell^\infty})_+\ni\mathbf{x}_\alpha\uparrow$ then
$$
   (B_{\ell^\infty})_+\ni S\mathbf{x}_\alpha\uparrow\mathbf{y}:=
   \sum_{n=2}^\infty\Bigl(\sum_{k=1}^{n-1}\frac{1}{2^k}\sup_\alpha[x_\alpha]_k\Bigl) e_n. 
$$
It is straightforward to see
$\|S\mathbf{x}_\alpha-\mathbf{y}\|_\infty\to 0$.  
Therefore, $S\in\text{\rm L}_{\text{\rm KB}}(\ell^\infty)$.

\medskip
\noindent
Consider the sequence
$\mathbf{b}_n=\sum_{k=1}^n e_k+\frac{1}{2}\sum_{k=n+1}^\infty e_k\in\ell^\infty$.
Then $\mathbf{b}_n\uparrow\mathbb{1}=\sum_{k=1}^\infty e_k$. 
Clearly, $T(\mathbf{b}_n)\equiv\frac{1}{2}e_1$ and
$$
   S(\mathbf{b}_n)\convsn\mathbf{u}=
   \left(0, \ \frac{1}{2}, \ \frac{1}{2}+\frac{1}{4}, \ 
   \frac{1}{2}+\frac{1}{4}+\frac{1}{8}, \ldots \right),
$$
and hence $\|(S+T)(\mathbf{b}_n)-\mathbf{v}\|_\infty\to 0$, where
$$
   \mathbf{v}=\left(\frac{1}{2}, \ \frac{1}{2}, \ \frac{1}{2}+\frac{1}{4}, \ 
   \frac{1}{2}+\frac{1}{4}+\frac{1}{8}, \ldots \right).
$$
Next, we show $\mathbf{v}\notin(S+T)E$. 

Assume in contrary, $(S+T)\mathbf{a}=\mathbf{v}$ for some $\mathbf{a}\in\ell^\infty$. Then 
$$
   \|(S+T)(\mathbf{b}_n)-(S+T)\mathbf{a}\|_\infty\to 0.
$$
In other words, 
$
   S(\mathbf{b}_n)+\frac{1}{2}e_1\convsn S\mathbf{a}+T\mathbf{a}= 
   S\mathbf{a}+f(\mathbf{a})e_1.
$
As $[S\mathbf{y}]_1=0$ for all $\mathbf{y}\in\ell^\infty$, we conclude
$$
   f(\mathbf{a})=\frac{1}{2}.
   \eqno(2)
$$
Since $S(\mathbf{b}_n)\convsn\mathbf{u}$ 
then $S\mathbf{a}=\mathbf{u}$, 
and hence 
$$
   \left(0,\ \frac{1}{2}a_1, \ \frac{1}{2}a_1+\frac{1}{4}a_2, \
   \frac{1}{2}a_1+\frac{1}{4}a_2+\frac{1}{8}a_3, \ldots\right)=
$$
$$
   \left(0, \ \frac{1}{2}, \ \frac{1}{2}+\frac{1}{4}, \ 
   \frac{1}{2}+\frac{1}{4}+\frac{1}{8}, \ldots \right).
$$
By recursive computation, we obtain $\mathbf{a}=\left(1,1, 1, 1, \ldots\right)$.
Then $f(\mathbf{a})=1$, which is violation (2). The obtained contradiction 
proves $\mathbf{v}\notin(S+T)E$.

We have shown that there is no element $\mathbf{a}\in\ell^\infty$ satisfying
$\|(S+T)(\mathbf{b}_n)-(S+T)\mathbf{a}\|_\infty\to 0$. Therefore, 
$S+T\notin\text{\rm L}_{\text{\rm KB}}(\ell^\infty)$.

It is worth to note that $S$ lies in the closure of $\text{\rm FR}(\ell^\infty)$
in the operator norm, and hence $S\in\text{\rm K}(\ell^\infty)$. Indeed, 
$S=\sum_{m=1}^\infty S_m$, where 
$S_m\mathbf{a}=\frac{a_m}{2^m}\sum_{i=m+1}^\infty e_i$, $\|S_m\|=\frac{1}{2^m}$.

Also note that the operators $S$ and $T$ can be considered as operators belonging to 
$\text{\rm K}(\ell^\infty,c)$ or to $\text{\rm K}(c)$.
}
\end{exam}

\subsection{}
It is worth to note that, for an operator $T\in\text{\rm K}_+(E)$ 
and a bounded increasing net $(x_\alpha)$ in $E_+$, 
a vector $x\in E$ satisfying $\|Tx_\alpha-Tx\| \to 0$ doesn't have to 
be an upper bound of $(x_\alpha)$. Indeed, let $l$ be a positive 
functional on $c$ that assigns to any element of $c$ its limit. 
Define $T:c\to c$ by $Tx:=l(x)\cdot\sum_{k=1}^\infty e_k$.
By Lemma \ref{prop1}, $T$ is a \text{\rm KB} operator.
Take $x_n:=\sum_{k=1}^n e_k\in c$.
Then $0\le x_n\uparrow$ and $\|x_n\|\equiv 1$.
Yet, for each $x\in c$ satisfying $x_n\uparrow \le x$, we have
$\|Tx_n-Tx\|=\|Tx\| \ge 1$ for all $n\in\mathbb{N}$.

\medskip
The next result develops the idea of Example~\ref{Example1 c_0}.

\begin{theorem}\label{prop3}
In a Banach lattice $E$ with order continuous norm 
the following conditions are equivalent. 
\begin{enumerate}[$i)$]
\item 
$E$ is a \text{\rm KB}-space.
\item  
Every bounded operator on $E$ is \text{\rm KB}.
\item  
Every positive operator on $E$ is \text{\rm KB}.
\item  
Every positive weakly compact operator on $E$ is \text{\rm KB}.
\item  
Every positive compact operator on $E$ is \text{\rm KB}.
\end{enumerate}
\end{theorem}

\begin{proof}
Implications $i)\Longrightarrow ii)\Longrightarrow iii)\Longrightarrow iv)\Longrightarrow v)$ 
are obvious.

$v) \Longrightarrow i)$ \ 
Assume, in contrary, that $E$ is not a \text{\rm KB}-space. 
Then there is a homeomorphic lattice embedding 
$c_0 \stackrel{i}{\longrightarrow}E$, say 
$i\Bigl(\sum\limits_{k=1}^\infty a_k \cdot e_k\Bigl)=\sum\limits_{k=1}^\infty a_k \cdot u_k$.
WLOG, assume $(u_k)$ to be a disjoint normalized sequence in $E_+$. 
Let $f_k$ be norm-one functionals on $E$ such that $f_k(u_k)=1$ for all $k\in\mathbb{N}$.
Then $|f_k|$ are norm-one positive functionals with 
$|f_k|(u_k)=1$ for all $k$.
Indeed, $\||f_k|\|=\|f_k\|=1$ and 
$$
   1=|f_k(u_k)|\le|f_k|(u_k)=\sup\{|f_k(x)|: |x|\le u_k\}\le
$$
$$
   \sup\{|f_k(x)|:\|x\|\le 1\}=\|f_k\|=1 \ \ \ \ \ \ (k\in\mathbb{N}).
$$
Let $P_k$ be the band projection onto the principle 
band $B_{u_k}$ of $E$ generated by $u_k$. Define an
operator $T\in\text{\rm L}_+(E)$ by
$$
   Tx=\sum_{k=1}^\infty\frac{|f_k|(P_kx)}{k}u_k . 
   \eqno(3)
$$
Since $||f_k|(P_kx)|\le\||f_k|\|\cdot\|P_kx\|\le\|x\|$ for all $k\in\mathbb{N}$,
the formula (3) implies
$$
   T(B_E)\subseteq i([-z,z]),  
   \eqno(4)
$$
where $z:=\sum\limits_{k=1}^\infty \frac{1}{k}\cdot e_k\in c_0$.
As order intervals of $c_0$ are compact,
the set $i([-z,z])$ is compact in $E$.
It follows from (4) that $T$ is a compact operator. 
By the condition~$v)$, $T$ is a \text{\rm KB} operator. 
Consider an increasing bounded sequence $y_n=\sum\limits_{k=1}^n u_k$ in $E$.
Then there is some $y\in E$ such that $\|Ty_n-Ty\|\to 0$. Since 
$$
   Ty_n=\sum_{k=1}^n\frac{|f_k|(P_ky_n)}{k}u_k=
   \sum_{k=1}^n\frac{|f_k|(u_k)}{k}u_k=
   \sum_{k=1}^n\frac{1}{k}u_k 
$$
for all $n\in\mathbb{N}$, $Ty=\sum_{k=1}^\infty\frac{1}{k}u_k$. 
Let $w_n=\sum_{k=1}^n P_ky$. Then $w_n\convo y$.
Since the norm in $E$ is order continuous, $\|w_n-y\|\to 0$, and hence
$(w_n)$ is norm-Cauchy.
It follows from $\|Ty_n-Ty\|\to 0$, $Ty_n=\sum_{k=1}^n\frac{1}{k}u_k$, and
$Ty=\sum_{k=1}^\infty\frac{1}{k}u_k$ that $|f_n|(P_ny_n)=|f_n|(P_ny)$.
Then
$$
   1=|f_n|(u_n)=|f_n|(P_ny_n)=|f_n|(P_ny)\le\|P_ny\|,
$$
violating $\|P_ny\|=\|w_{n}-w_{n-1}\|\to 0$.
The obtained contradiction completes the proof.
\end{proof}
\noindent
The author doesn't know whether the condition of order continuity 
of the norm in Theorem \ref{prop3} can be weakened.
\medskip

\begin{theorem}\label{cor1}
Let $E$ be a Banach lattice with order continuous norm.
The following conditions are equivalent. 
\begin{enumerate}[$i)$]
\item 
$E$ is a \text{\rm KB}-space.
\item 
$\text{\rm L}(E,Y)=\text{\rm L}_{\text{\rm KB}}(E,Y)$ for each $Y$.
\item 
$\text{\rm L}_{\text{\rm qKB}}(E,Y)=\text{\rm L}_{\text{\rm KB}}(E,Y)$ for each $Y$.
\item  
$\text{\rm L}_{\text{\rm qKB+}}(E,F)=\text{\rm L}_{\text{\rm KB+}}(E,F)$ for each $F$.
\item  
$\text{\rm L}_{\text{\rm qKB+}}(E)=\text{\rm L}_{\text{\rm KB+}}(E)$.
\item 
$\text{\rm K}_+(E)\subseteq\text{\rm L}_{\text{\rm KB}}(E)$.
\end{enumerate}
\end{theorem}

\begin{proof}
$i) \Longrightarrow ii) \Longrightarrow iii) \Longrightarrow iv)\Longrightarrow v)$ \
is obvious.\\
$v) \Longrightarrow vi)$ \
Let $T\in\text{\rm K}_+(E)$. Then
$T\in\text{\rm L}_{\text{\rm qKB+}}(E)$ by Lemma \ref{prop2},
and hence 
$T\in\text{\rm L}_{\text{\rm KB}}(E)$.\\
$vi) \Longrightarrow i)$ follows from Theorem \ref{prop3}.
\end{proof}

\subsection{}
By \cite[Thm.2.6]{AlEG} quasi-KB operators satisfy the domination property,
that is, for each positive quasi-KB operator $S$, every operator $T$ such 
that $0\le T\le S$ is also quasi-KB. The following example (which is motivated by
Example \ref{Example1 c_0}) shows that generally KB operators do not
satisfy the domination property.

\begin{exam}\label{domination for KB fails}
{\em
Define operators $S,T\in\text{\rm L}(c)$ by
$$
   S\left(\sum_{n=1}^\infty a_n e_n\right)=\sum_{n=1}^\infty \frac{a_n}{2^n} e_n, \ \ \text{\rm and}
$$
$$
   T\left(\sum_{n=1}^\infty a_n e_n\right)=
   \sum_{n=1}^\infty\Bigl(\sum_{k=1}^\infty\frac{a_k}{2^k}\Bigl) e_n.
$$
Clearly, $0\le S\le T$. The operator $T$ is of rank one, and hence  
$T\in\text{\rm L}_{\text{\rm KB}}(c_0)$ by Lemma \ref{prop1}. But  
$S\notin\text{\rm L}_{\text{\rm KB}}(c_0)$ by Example \ref{Example1 c_0}.
}
\end{exam}

\section{Levi operators}

\text{\rm Levi} operators appear as 
operator versions of Banach lattices with Levi norms. 
In this section, we study some relations between Levi operators and 
other operators on Banach lattices.

\subsection{}
Zhang and Chen proved recently in \cite[Thm.2.2]{ZC} 
that every \text{\rm Levi} operator in the sense \cite[Def.1.1]{ZC}
is order bounded (in their proof, they understand under the order convergence 
the so-convergence of Definition \ref{order convergence}). 
In our terminology, \cite[Thm.2.2]{ZC} tell us that 
$\text{\rm L}_{\text{\rm c-qLevi}}(E,F)\subseteq\text{\rm L}_r(E,F)$,
whenever $F$ is Dedekind complete. We extend the Zhang - Chen
result as follows.
  
\begin{theorem}\label{T5}
Every quasi Levi operator is order bounded.
\end{theorem}

\begin{proof}
Let $T\in\text{\rm L}_{\text{\rm qLevi}}(E,F)$,
and let $[0,b]$ be an order interval in $E$. 
Denote by $F^\delta$ the Dedekind completion 
of the vector lattice $F$. WLOG suppose that
$F$ is a vector sublattice of $F^\delta$. Since $F$ is mazorizing 
in $F^\delta$, it suffices to prove order boundedness of $T[0,b]$ 
in $F^\delta$. 

As in the proof of \cite[Thm.2.2]{ZC}, consider an increasing bounded net
in $E_+$ indexed by the upward directed set $[0,b]$ via
$b_\alpha=\alpha$ for each $\alpha\in [0,b]$. Since $T$ is quasi Levi,
the net $(Tb_\alpha)$ is \text{\rm o}-Cauchy in $F$, and hence in $F^\delta$.
As every Dedekind complete vector lattice is order complete,
there exists $z\in F^\delta$ satisfying $Tb_\alpha\convo z$ in $F^\delta$.

Since $(F^\delta)^\delta=F^\delta$ then $Tb_\alpha\convso z$ in $F^\delta$
by \cite[Prop.1.5]{AS}. Then $|T\alpha-z|=|Tb_\alpha-z|\le y_\alpha\downarrow 0$
for a net $(y_\alpha)$ in $F^\delta$. Thus, $|T\alpha-z|\le y_0$, and hence
$|T\alpha|\le|z|+y_0$ for all $\alpha\in [0,b]$. Therefore,
$T[0,b]\subseteq[-(|z|+y_0),(|z|+y_0)]$, as desired.
\end{proof}

\subsection{}
Our next result is a Levi-version of Theorem \ref{cor1}. 

\begin{theorem}\label{prop8}
For a Banach lattice $F$ with order continuous norm 
the following conditions are equivalent. 
\begin{enumerate}[$i)$]
\item 
$F$ is a \text{\rm KB}-space.
\item  
$\text{\rm L}_+(F)=\text{\rm L}_{\text{\rm Levi+}}(F)$.
\item 
$\text{\rm L}_{\text{\rm qLevi+}}(F)=\text{\rm L}_{\text{\rm Levi+}}(F)$.
\item 
$\text{\rm L}_{\text{\rm c-qLevi+}}(F)=\text{\rm L}_{\text{\rm Levi+}}(F)$.
\item 
$\text{\rm K}_+(F)\subseteq\text{\rm L}_{\text{\rm Levi}}(F)$.
\end{enumerate}
\end{theorem}

\begin{proof}
$i) \Longrightarrow ii)$.
By items $ii)$ and $iii)$ of Lemma \ref{prop2}, 
$$
   \text{\rm L}_+(F)\subseteq\text{\rm L}_{\text{\rm KB+}}(E,F)\subseteq
   \text{\rm L}_{\text{\rm Levi+}}(E,F)\subseteq\text{\rm L}_+(F).
$$

Implications $ii)\Longrightarrow iii)\Longrightarrow iv)$ are obvious.

$iv) \Longrightarrow v)$ follows from item $v)$ of Lemma \ref{prop2}.

$v) \Longrightarrow i)$.
First, we show $\text{\rm K}_+(F)\subseteq\text{\rm L}_{\text{\rm KB}}(F)$.
Let $T\in\text{\rm K}_+(F)$ and $(B_F)_+\ni x_\alpha\uparrow$.
By the assumption, there exists an $x\in F$ such that $Tx_\alpha\convo Tx$.
Since the norm in $F$ is order continuous, $Tx_\alpha\convn Tx$
and hence $T\in\text{\rm L}_{\text{\rm KB}}(F)$.

Now, Theorem \ref{cor1} implies that $F$ is a \text{\rm KB}-space.
\end{proof}

\subsection{}
By Example \ref{KB-not-compl},  $\text{\rm L}_{\text{\rm Levi}}(c_0)$ 
is not closed in $\text{\rm L}(c_0)$ under the operator norm.
In contrast to the fact that the quasi {\text{\rm KB}} operators are always
closed by Lemma \ref{L2}, the next example shows that the complete
quasi {\text{\rm Levi}} and quasi {\text{\rm Levi}} operators are not
in general.

\begin{exam}\label{qLevi-not-compl}
{\em 
Define $S_n\in\text{\rm L}_{\text{\rm FR}}\bigl((\oplus_{n=1}^{\infty}\ell^2_{2^n})_0\bigl)$ by
$$
   S_nx:=(2^{-\frac{1}{3}}T_1x_1,2^{-\frac{2}{3}}T_2x_2,\dots 2^{-\frac{n}{3}}T_nx_n,0,0,\dots),
$$
where $T_n:\ell^2_{2^n}\to\ell^2_{2^n}$ are taken as in \cite[Ex.5.6]{AB}.
Then, $S_n\in\text{\rm L}_{\text{\rm Levi}}\bigl((\oplus_{n=1}^{\infty}\ell^2_{2^n})_0\bigl)$
by Lemma \ref{prop1}, and hence operators $S_n$ are (complete) quasi {\text{\rm Levi}}.
Now, define an operator $S\in\text{\rm L}\bigl((\oplus_{n=1}^{\infty}\ell^2_{2^n})_0\bigl)$ by
$$
   Sx:=(2^{-\frac{1}{3}}T_1x_1,2^{-\frac{2}{3}}T_2x_2,\dots2^{-\frac{n}{3}}T_nx_n,\dots).
$$
Then $\|S_n-S\|\to 0$. The Banach lattice $(\oplus_{n=1}^{\infty}\ell^2_{2^n})_0$ 
has order continuous norm, and hence is Dedekind complete. However $|S|$ does not exist,
so $S$ is not order bounded. Thus, 
$S\notin\text{\rm L}_{\text{\rm qLevi}}\bigl((\oplus_{n=1}^{\infty}\ell^2_{2^n})_0\bigl)$
due to Theorem \ref{T5}.
}
\end{exam}

\medskip
Quasi-Levi operators satisfies the domination property by \cite[Thm.2.7]{AlEG}. 
The following example shows that, in general, neither
Levi nor $\sigma$-Levi operators satisfy the domination property.

\begin{exam}\label{domination for Levi fails}
{\em
Consider operators $0\le S\le T\in\text{\rm L}(c)$ as in Example \ref{domination for KB fails}.
Lemma \ref{prop1} implies that $T\in\text{\rm L}_{\text{\rm Levi}}(c_0)$.
By Example \ref{Example1 c_0}, 
$S\in\text{\rm L}_{\text{\rm c-qLevi}}(c_0)\setminus\text{\rm L}^\sigma_{\text{\rm Levi}}(c_0)$.
}
\end{exam}
\noindent
The author does not know where or not the complete quasi-Levi operators 
satisfy the domination property.

\subsection{}
We conclude the paper with a modification of Example \ref{KB are not VS} 
which shows that \text{\rm Levi} operator may also be not stable 
under positive rank one perturbations.

\begin{exam}\label{Levi are not VS}
{\em
Define positive \text{\rm KB} operators $S,T\in\text{\rm L}(\ell^\infty)$ 
as in Example \ref{KB are not VS}. Then,
$S,T\in\text{\rm L}_{\text{\rm Levi}}(\ell^\infty)$ by Lemma \ref{prop2}.
Take again the increasing bounded sequence
$$
   \mathbf{b}_n=\sum_{k=1}^n e_k+\frac{1}{2}\sum_{k=n+1}^\infty e_k\in\ell^\infty.
$$
Then $T(\mathbf{b}_n)\equiv\frac{1}{2}e_1$ and $S(\mathbf{b}_n)\convsn\mathbf{u}$,
and hence $S(\mathbf{b}_n)\convo\mathbf{u}$, where
$$
   \mathbf{u}=
   \left(0, \ \frac{1}{2}, \ \frac{1}{2}+\frac{1}{4}, \ 
   \frac{1}{2}+\frac{1}{4}+\frac{1}{8}, \ldots \right),
$$
and hence $((S+T)(\mathbf{b}_n)-\mathbf{v})\convo 0$, where
$$
   \mathbf{v}=\left(\frac{1}{2}, \ \frac{1}{2}, \ \frac{1}{2}+\frac{1}{4}, \ 
   \frac{1}{2}+\frac{1}{4}+\frac{1}{8}, \ldots \right).
$$
Would be $\mathbf{v}\in(S+T)E$, say $(S+T)\mathbf{a}=\mathbf{v}$ 
for some $\mathbf{a}\in\ell^\infty$, then 
$(S+T)(\mathbf{b}_n)\convo(S+T)\mathbf{a}$.
In this case, 
$$
   S(\mathbf{b}_n)+\frac{1}{2}e_1\convo S\mathbf{a}+T\mathbf{a}= 
   S\mathbf{a}+f(\mathbf{a})e_1,
$$
where $f$ is a Banach limit as in Example \ref{KB are not VS}.
Again, we conclude $f(\mathbf{a})=\frac{1}{2}$.
Using $S(\mathbf{b}_n)\convo\mathbf{u}$, we obtain 
$S\mathbf{a}=\mathbf{u}$, 
and hence 
$$
   \left(0,\ \frac{1}{2}a_1, \ \frac{1}{2}a_1+\frac{1}{4}a_2, \
   \frac{1}{2}a_1+\frac{1}{4}a_2+\frac{1}{8}a_3, \ldots\right)=
$$
$$
   \left(0, \ \frac{1}{2}, \ \frac{1}{2}+\frac{1}{4}, \ 
   \frac{1}{2}+\frac{1}{4}+\frac{1}{8}, \ldots \right).
$$
By recursive computation, we have $\mathbf{a}=\left(1,1, 1, 1, \ldots\right)$.
Then $f(\mathbf{a})=1$, violating previously obtained 
$f(\mathbf{a})=\frac{1}{2}$. The obtained contradiction 
proves $\mathbf{v}\notin(S+T)E$. Thus,  
$S+T\notin\text{\rm L}_{\text{\rm Levi}}(\ell^\infty)$.
}
\end{exam}

{
}
\end{document}